\documentclass{article}

\usepackage[utf8]{inputenc} 
\usepackage[T1]{fontenc}    
\usepackage{hyperref}       
\usepackage{url}            
\usepackage{booktabs}       
\usepackage{amsfonts}       
\usepackage{nicefrac}       
\usepackage{microtype}      

\usepackage{graphicx}
\usepackage{amsmath,amssymb,amsfonts}
\usepackage{mathrsfs}
\usepackage{graphics}
\usepackage{lipsum, color}
\usepackage{epstopdf}
\usepackage{amsthm}

\usepackage{hhline}
\usepackage{booktabs}
\usepackage{makecell}
\usepackage{multirow}
\usepackage{dcolumn}  
\usepackage{lineno}  
\usepackage{blindtext}  
\usepackage{verbatim} 
\usepackage{listings} 
\usepackage{color}    
\usepackage{tikz}

\usetikzlibrary{shapes,arrows}
\usetikzlibrary{positioning}
\usetikzlibrary{shapes,arrows}
\usetikzlibrary{arrows,calc, matrix}
\tikzset{
    block/.style = {draw, rectangle, 
        minimum height=1cm, 
        minimum width=1cm},
    input/.style = {coordinate,node distance=1.3cm},
    output/.style = {coordinate,node distance=2.3cm},
    arrow/.style={draw, -latex,node distance=2cm},
    pinstyle/.style = {pin edge={latex-, black,node distance=2cm}},
    sum/.style = {draw, circle, node distance=1cm}
}

\DeclareMathOperator*{\argmin}{arg\,min}

\newcommand{\bR}{\mathbb{R}}
\newcommand{\bE}{\mathbb{E}}

\DeclareMathOperator*{\diag}{diag}
\DeclareMathOperator*{\tr}{tr}
\usepackage{forest}

\newtheorem{theorem}{Theorem}
\newtheorem{definition}{Definition}
\newtheorem{lemma}{Lemma}

\newtheorem{corollary}{Corollary}

\newtheorem*{theorem*}{Theorem}
\usepackage{color}

\begin{document}

\title{\Large \bf Escaping Locally Optimal Decentralized Control Polices via Damping\thanks{Email: han\_feng@berkeley.edu, lavaei@berkeley.edu}}
\author{
  Han Feng and Javad Lavaei 
\thanks{ 
This work was supported by grants from ARO, ONR, AFOSR, and NSF.}}
\maketitle


\begin{abstract}
We study the evolution of locally optimal decentralized controllers with the damping of the control system. Empirically it is shown that even for instances with an exponential number of connected components, damping merges all local solutions to the one global solution. We characterize the evolution of locally optimal solutions with the notion of hemi-continuity and further derive asymptotic properties of the objective function and of the locally optimal controllers as the damping becomes large. Especially, we prove that with enough damping, there is no spurious locally optimal controller with favorable control structures. The convoluted behavior of the locally optimal trajectory is illustrated with numerical examples. 
\end{abstract}

\section{Introduction} 

The optimal decentralized control problem (ODC) adds controller constraints to the classical centralized optimal control problem. This addition breaks down the separation principle and the classical solution formulas culminated in~\cite{Doyle1989}. Although ODC has been proved intractable in general~\cite{Witsenhausen1968, Blondel2000}, the problem has convex formulations under assumptions such as partially nestedness~\cite{Shah2013}, positiveness~\cite{Rantzer2015}, and quadratic invariance~\cite{Lessard2014}. A recently proposed System Level Approach~\cite{Wang2017} convexified the problem in the space of system response matrix. Convex relaxation techniques have been extensively documented in \cite{Boyd1994}, though it is considered challenging to solve large scale optimization problems with linear matrix inequalities. 

The line of research on convexification is in contrast with the success of stochastic gradient descent well-documented in machine learning practice~\cite{hardtTrainFasterGeneralize2015,Goodfellow-et-al-2016}. Admittedly, the problem of generalizability, training speed, and fairness in machine learning departs from the traditional control focus on stability, robustness, and safety. Nevertheless, the interplay of the two has inspired fruitful results. As an example, to solve the linear-quadratic optimal control problem, the traditional nonlinear programming methods include Gauss-Newton, augmented Lagrangian, and Newton's methods~\cite{levineDeterminationOptimalConstant1970,Wenk1980,linAugmentedLagrangianApproach2011,makilaComputationalMethodsParametric1987}. Only in the last few years do researchers started to look at the classical problem with the newly developed optimization techniques and proved the efficiency of policy gradient methods in model-based and model-free optimal control problems~\cite{fazelGlobalConvergencePolicy2018}. This efficiency statement of local search, however, is unlikely to carry over trivially to ODC, due to the NP-hardness of the problem and the recent investigation of the topological properties of ODC in~\cite{fengExponentialNumberConnected}. 
Nevertheless, questions can be answered without contradicting the general complexity statement. For example, it is known that damping of the system reduces the number of connected components of the set of stabilizing decentralized controllers. Does damping reduce the number of locally optimal decentralized controllers? 
This paper attempts an answer with (1) a study of the continuity properties of the trajectories of the locally optimal solutions formed by varying damping, and (2) an asymptotic analysis of the trajectories as the damping becomes large. The observation of our study shall shed light on the properties of local minima in reinforcement learning, whose aim is to design optimal control policies and different local minima have different practical behaviors. 


This work is closely related to continuation methods such as homotopy. They are known to be appealing yet theoretically poorly understood~\cite{mobahiTheoreticalAnalysisOptimization2015}. Homotopy has been used as an initialization strategy in optimal control: in \cite{broussardActiveFlutterControl1983}, the author mentioned the idea of gradually moving from a stable system to the original system to obtain a stabilizing controller. The paper \cite{zigicHomotopyApproachesH21991} considered $H_2$-reduced order problem and proposed several homotopy maps and initialization strategies; in its numerical experiments, initialization with a large multiple of $-I$ was found appealing. \cite{emmanuelg.collinsjr.ComparisonDescentContinuation1998} compared descent and continuation algorithms for $H_2$ optimal reduced-order control problem and concluded that homotopy methods are empirically superior to descent methods. The difficulty of obtaining a convergence theory for general constrained optimal control problem can be appreciated from the examples in \cite{mercadalHomotopyApproachOptimal1991}. Compared with those earlier works, we consider a special kind of continuation, that is, damping,  to improve the locally optimal solutions in optimal decentralized control. Our focus is not so much on following a specific path but on the evolution of several paths and the movement of locally optimal solutions from one path to another. 

The remainder of this paper is organized as follows. 
Notations and problem formulations are given in Section~\ref{sec:formulation}. Continuity and asymptotic properties of our damping strategies are outlined in Section~\ref{sec:continuity} and Section~\ref{sec:asymptotic}, respectively. Numerical experiments are detailed in Section~\ref{sec:numerical}. Concluding remarks are drawn in Section~\ref{sec:conclusion}.

\section{Problem Formulation}\label{sec:formulation}

Consider the linear time-invariant system 
\begin{align*}
    \dot x(t) &= A x(t) + B u(t), 
\end{align*}
where $A\in \bR^{n\times n}$ and $B\in \bR^{n\times m}$ are real matrices of compatible sizes. The vector $x(t)$ is the state of the system with an unknown initialization $x(0)=x_0$, where $x_0$ is modeled as a random variable with zero mean and a positive definite covariance $\bE[x(0)x(0)^\top ] = D_0$. 
The control input $u(t)$ is to be determined via a static state-feedback law $u(t) = Kx(t)$ with the gain $K\in \bR^{m\times n}$ such that some quadratic performance measure is maximized. 
Given a controller $K$, the closed-loop system is \begin{align*}
\dot x(t) &= (A + BK)x(t). 
\end{align*}
A matrix is said to be stable if all its eigenvalues lie in the open left half plane. The controller $K$ is said to stabilize the system if $A+BK$ is stable. 
ODC optimizes over the set of structured stabilizing controllers \begin{align*}
\{K: A+BK \text{ is stable}, K\in \mathcal{S}\},
\end{align*}
where $\mathcal{S}\subseteq \bR^{m\times n}$ is a linear subspace of matrices, often specified by fixing certain entries of the matrix to zero. In that case, the sparsity pattern can be equivalently described with the indicator matrix $I_{\mathcal{S}}$, whose $(i,j)$-entry is defined to be 
\begin{align*}
[I_{\mathcal{S}}]_{ij}= 
\begin{cases} 1, \qquad \text{if $K_{ij}$ is free}\\
0, \qquad \text{if $K_{ij}=0$.}
\end{cases}
\end{align*}   

The structural constraint $K \in \mathcal{S}$ is then equivalent to $K \circ I_\mathcal{S}=K$, where $\circ$ denotes entry-wise multiplication. In the following, we will consider the discounted, or damped cost, which is defined as 
\begin{equation}
\begin{aligned}
    J(K, \alpha) = & \bE \int_0^\infty \left[e^{-2\alpha t}\left(  \hat x^\top (t) Q \hat x(t) + \hat u^\top  (t) R \hat u(t) \right)\right] dt \\ 
  s.t.  \quad   & \hat{\dot {x}}(t) =  A \hat x(t) + B \hat u(t) \\ 
  & \hat u(t) = K \hat x(t). 
\end{aligned}\label{eq:damped-rep}
\end{equation}
where $Q \succeq 0$ is positive semi-definite and $R\succ 0$ is positive definite. The expectation is taken over $x_0$. 
Setting $x(t) = e^{-\alpha t} \hat x(t), u(t) = e^{-\alpha t} \hat u(t)$, the cost $J(K, \alpha)$ can be equivalently written as  
\begin{align}
\begin{aligned}
J(K, \alpha) = & \bE \int_0^\infty \left[ x^\top (t) Q x(t) + u^\top  (t) R u(t) \right] dt \\ 
  s.t.  \quad   & \dot x(t) = (A - \alpha I) x(t) + Bu(t) \\ 
  & u(t) = K x(t), 
  \end{aligned}\label{eq:damped-matrix-rep}
\end{align}

The two equivalent formulations above motivate the notion of ``damping property''. We make a formal statement below. 
\begin{lemma}\label{def:damping}
  The function $J(K, \alpha)$ defined in \eqref{eq:damped-rep} and \eqref{eq:damped-matrix-rep} satisfies the following  ``damping property'': suppose that $K$ stabilizes the system $(A-\alpha I, B)$, then for all $\beta > \alpha$, $K$ stabilizes the system $(A-\beta I, B)$ with $J(K, \beta) < J(K, \alpha)$. 
\end{lemma}
\begin{proof}
  From the formulation \eqref{eq:damped-odc}, when $A-\alpha I +BK$ is stable and  $\beta>\alpha$, it holds that $A-\beta I +BK = (A-\alpha I + BK) - (\beta - \alpha) I$ is stable. Therefore, $J(K, \beta)$ is well-defined. From formulation \eqref{eq:damped-rep}, $J(K, \beta) < J(K, \alpha)$. 
\end{proof}
The ODC problem can be succinctly written as 
\begin{equation}
\begin{aligned} \label{eq:damped-odc}
\min \quad &  J(K, \alpha) \\
s.t. \quad & K\in \mathcal{S} \\ 
           & A - \alpha I + BK \text{ is stable}. 
\end{aligned}
\end{equation}
We denote its set of globally optimal controllers by $K^*(\alpha)$, and its set of locally optimal controllers by $K^\dagger(\alpha)$. The paper studies the properties of $K^*(\alpha)$, $K^\dagger(\alpha)$, and $J(K, \alpha)$ for $K\in K^*(\alpha)$ or $K^\dagger(\alpha)$. 

To motivate the study of $K^\dagger(\alpha)$, consider Figure~\ref{fig:expeg} below. The set-up of the experiments will be detailed in Section~\ref{sec:numerical}. It is known that systems of this type have a large number of locally optimal controllers~\cite{fengExponentialNumberConnected}. The left figure plots selected trajectories of $J(K, \alpha)$ against $\alpha$, where $K\in K^\dagger(\alpha)$. The selected trajectories are connected to a stabilizing controller in $K^\dagger(0)$. The lowest curve corresponds to $J(K^*(\alpha), \alpha)$.  The right figure plots the distance of the selected $K\in K^\dagger(\alpha)$ to the one $K\in K^*(\alpha)$. 

\begin{figure}[htbp]
  \centering
  \includegraphics[width=0.49\textwidth]{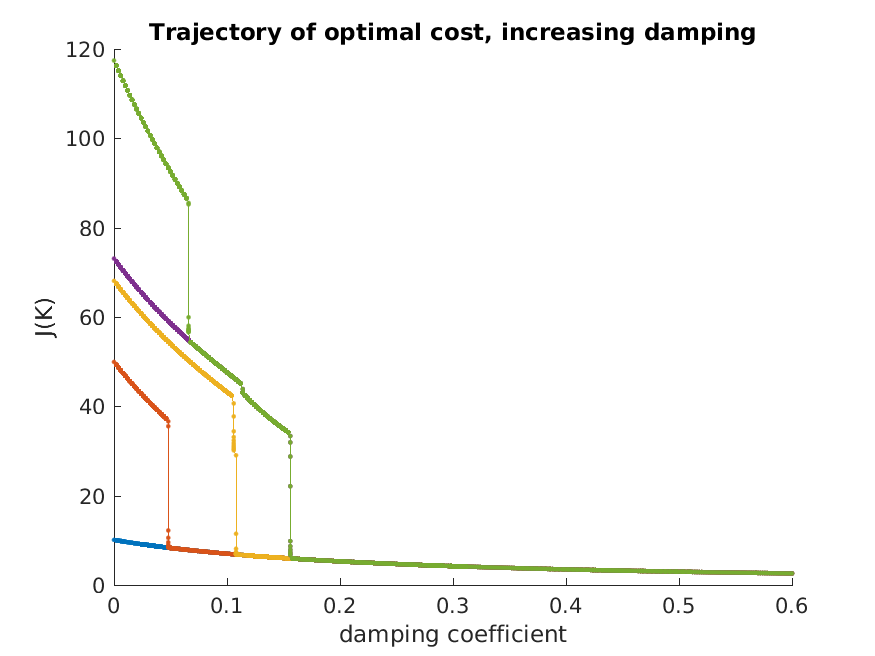}
  \includegraphics[width=0.49\textwidth]{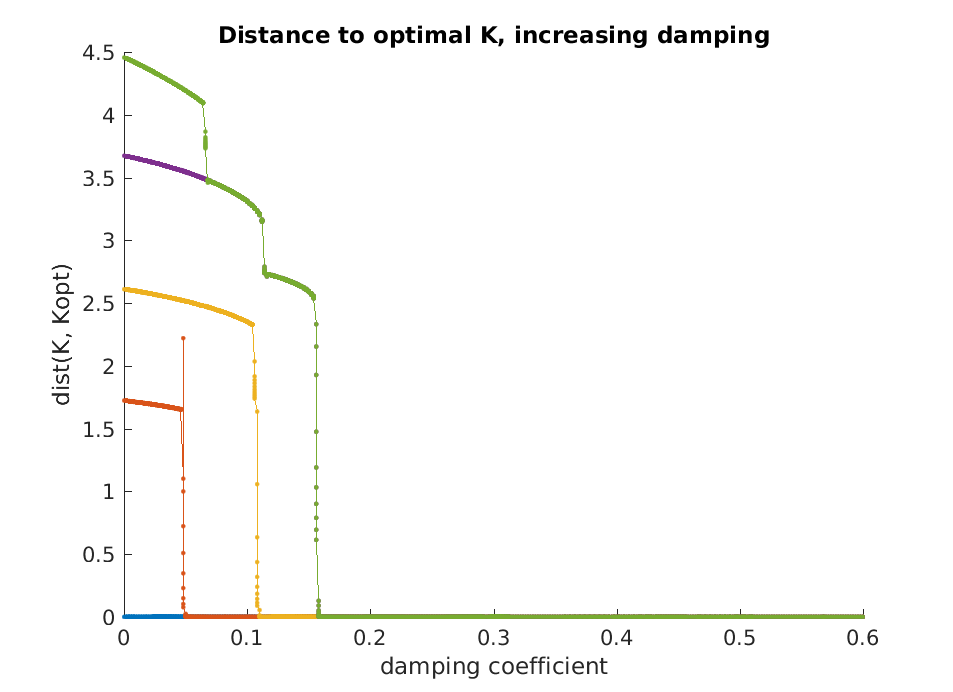}
  \caption{\label{fig:expeg}Trajectory of system in equation \eqref{eq:aeg}}
\end{figure}    

The fact that modest damping causes the locally optimal trajectories to ``collapse'' to each other is a very attractive phenomenon. Especially, they suggest two improving heuristics. 
\begin{itemize}
  \item Solve~\eqref{eq:damped-odc} from a large $\alpha$ and then gradually decrease $\alpha$ to $0$. 
  \item Start from a locally optimal $K\in K^\dagger(\alpha)$, solve~\eqref{eq:damped-odc} while gradually increase $\alpha$ to a positive value and then decrease $\alpha$ to $0$. 
\end{itemize}
The first idea shall avoid many unnecessary local optimum and its empirical behavior has been documented in \cite{zigicHomotopyApproachesH21991}. The second idea has the potential to improve the locally optimal controllers obtained from many other methods. Due to the NP-hardness of general ODC, we expect no guarantee of producing a globally optimal, or even a stabilizing, decentralized controller. The breakdown of these heuristics will be discussed in Section~\ref{sec:numerical}.

\section{Continuity}\label{sec:continuity}
This section studies the continuity properties of $K^*(\alpha)$ and $K^\dagger(\alpha)$. The key notion of hemi-continuity captures the evolution of parametrized optimization problems. 
\begin{definition}
  The set valued map $\Gamma: \mathcal{A}\to \mathcal{B}$ is said to be upper hemi-continuous (uhc) at a point $a$ if for any open neighborhood $V$ of $\Gamma(a)$ there exists a neighborhood $U$ of $a$ such that $\Gamma(U)\subseteq V$. 
\end{definition}
A related notion of lower hemi-continuity is provided in the supplement. A set-valued map is said to be continuous if it is both upper and lower hemi-continuous. A single-valued function is continuous if and only if it is uhc. We restate a version of Berge Maximum Theorem with a compactness assumption from~\cite{okRealAnalysisEconomic2007}. 
\begin{lemma}[Berge Maximum Theorem]
Let $\mathcal{A}\subseteq \bR$ and $\mathcal{S}\subseteq \bR^{m\times n}$, assume that $J: \mathcal{S}\times \mathcal{A} \to \bR$ is jointly continuous and $\Gamma: \mathcal{A} \to \mathcal{S}$ is a compact-valued correspondence. Define
\begin{align}\label{eq:parametric-opt}
K^*(\alpha) = \argmin \{J(K,\alpha) | K \in \Gamma(\alpha)\}, \text{ for all } \alpha \in \mathcal{A}, 
\end{align}and
  \[ J(K^*(\alpha), \alpha) = \min \{J(K,\alpha) | K \in \Gamma(\alpha)\}, \text{ for all }\alpha \in \mathcal{A}.\]
If $\Gamma$ is continuous at some $\alpha\in \mathcal{A}$, then $J(K^*(\alpha), \alpha)$ is continuous at $\alpha$. Furthermore, $K^*$ is non-empty, compact-valued, closed, and upper hemi-continuous.  
\end{lemma}
Berge Maximum Theorem does not trivially apply to ODC: the set of stabilizing controllers is open and often unbounded. However, a lower-level set trick applies. 

\begin{theorem}\label{thm:decrease}
 Assume that $K^*(0)$ is non-empty, then the set $K^*(\alpha)$ is non-empty for all $\alpha>0$. $K^*(\alpha)$ is upper hemi-continuous and the optimal cost $J(K^*(\alpha), \alpha)$ is continuous and strictly decreasing in $\alpha$. 
\end{theorem}
\begin{proof}
When $K^*(0)$ is non-empty, there is an optimal decentralized controller for the undamped system. With the set of stabilizing controller non-empty, we incur the ``damping property'' in Lemma~\ref{def:damping} and conclude 
  \begin{align*}
  J(K^*(\alpha), \alpha) \leq J(K^*(0), \alpha) < J(K^*(0), 0). 
  \end{align*}
The inequality above assumed existence of the globally controller for all values of damping parameter $\alpha$. This is true because the lower-level set of $J(K, \alpha)$ is compact~\cite{toivonenGloballyConvergentAlgorithm1985}. Precisely, define $\Gamma_M(\alpha)$ to be
  \begin{align*}
  \Gamma_M(\alpha) = \{K\in S: A-\alpha I + BK \text{ is stable and } J(K, \alpha) \leq M\}. 
  \end{align*}
  The set-valued function $\Gamma_M$ is compact-valued for all fixed $\alpha$ given a fixed $M$.
  From the damping property, we can select any $M > J(K^*(0),0)$ and optimize instead over $\Gamma_M(\alpha)$ without losing any globally optimal controller. The continuity of $\Gamma_M(\alpha)$ at $\alpha$ for almost all $M$ is proved in the supplement. Berge maximum theorem then applies and yields the desired continuity of $K^*(\alpha)$ and $J(K^*(\alpha), \alpha)$. 
\end{proof}
The argument above can be extended to characterize all locally optimal controllers. A caveat is the possible existence of locally optimal controllers with unbounded cost. Their existence does not contradict the damping property --- damping can introduce locally optimal controllers that are not stabilizing without the damping. 
\begin{theorem}
  Assume that $K^\dagger(0)$ is non-empty, then the set $K^\dagger(\alpha)$ is nonempty for all $\alpha>0$. Suppose furthermore that at an $\alpha_0 > 0$
  \[ \lim_{\epsilon\to 0^+ } \sup_{\alpha \in [\alpha_0-\epsilon, \alpha_0+\epsilon]} \sup_{K\in K^\dagger(\alpha)} J(K, \alpha) < \infty, \] 
  then $K^\dagger(\alpha)$ is upper hemi-continuous at $\alpha_0$ and the optimal cost $J(K^\dagger(\alpha), \alpha)$ is upper hemi-continuous at $\alpha_0$. 
\end{theorem}
\begin{proof}
That $K^\dagger(\alpha)$ is non-empty follows from the existence of globally optimal controllers in Theorem~\ref{thm:decrease}. Consider the parametrized optimization problem 
\begin{align}
 \min \quad & \|\nabla J(K, \alpha)\| \nonumber\\ 
   s.t. \quad & K \in \Gamma_M(\alpha) \label{eq:gradzero}. 
 \end{align} 
 The assumption ensures the existence of an $M$ and an $\epsilon>0$ such that $M > J(K, \alpha)$ for $K\in K^\dagger(\alpha)$ where $\alpha \in [\alpha_0-\epsilon, \alpha_0 + \epsilon]$. This choice of $M$ guarantees that the formulation \eqref{eq:gradzero} does not cut off any locally optimal controllers. As proved in the supplement, $\Gamma_M(\alpha)$ is continuous at $\alpha_0$ for almost any $M$, and a large $M$ can be selected to make $\Gamma_M(\alpha)$ continuous at $\alpha_0$. Berge Maximum Theorem applies to conclude that $K^\dagger(\alpha)$ is upper hemi-continuous. Since $J(K, \alpha)$ is jointly continuous in $(K, \alpha)$, $J(K^\dagger(\alpha), \alpha)$ is upper hemi-continuous. 
\end{proof}


\section{Asymptotic Properties}\label{sec:asymptotic}
In this section, we state asymptotic properties of the local solutions $K^\dagger(\alpha)$. The controllers $K\in K^\dagger(\alpha)$ satisfy the first order necessary conditions in the following equations \eqref{eq:fonp}-\eqref{eq:sparsity}; their derivation can be found in \cite{rautertComputationalDesignOptimal1997}. 
\begin{align}
& (A - \alpha I + BK)^\top  P_\alpha(K) +P_\alpha(K) (A - \alpha I + BK) + K^\top  RK + Q = 0 \label{eq:fonp}\\
& L_\alpha(K)(A-\alpha I  + BK)^\top  + (A-\alpha I +BK)L_\alpha(K) + D_0 = 0 \label{eq:fonl} \\ 
&  ((B^\top P_{\alpha}(K) + RK) L_\alpha(K))\circ I_S= 0\label{eq:stationary}\\
& K \circ I_\mathcal{S}=K \label{eq:sparsity}. 
\end{align}
The above conditions provide a closed-form expression of the cost 
\begin{equation}\label{eq:jpk}
  J(K, \alpha) = \tr(D_0 P_\alpha(K)). 
\end{equation}
It is worth pointing out that equations~\eqref{eq:fonp}-\eqref{eq:jpk} are algebraic, involving only polynomial functions of the unknown matrices $K, P_\alpha$ and $L_\alpha$. The matrices $P_\alpha$ and $L_\alpha$ are  written as a function of $K$ because they are uniquely determined from~\eqref{eq:fonp} and \eqref{eq:fonl} given a stabilizing controller $K$. 
The following theorem characterizes the evolution of locally optimal controllers for a specific sparsity pattern. The theorem justifies the practice of random initialization around zero. 

\begin{theorem}\label{thm:converge-to-zero}
  Suppose that the sparsity pattern $I_S$ is block-diagonal with square blocks and that $R$ has the same sparsity pattern as $I_S$. Then, all points in $K^\dagger$ converge to the zero matrix as $\alpha \to \infty$. Furthermore, $J(K, \alpha)\to 0$ as $\alpha\to\infty$ for all $K\in K^\dagger(\alpha)$. 
\end{theorem}

Not only do all locally optimal controllers approach zero, the problem is in fact convex over bounded regions with enough damping. 
\begin{theorem}\label{thm:convex}
For any given $r>0$, the Hessian matrix $\nabla^2 J(K, \alpha)$ is positive definite over $\|K\| \leq r$ for all large $\alpha$. 
\end{theorem} 

The proof of the two theorems above is given in the supplement.

\begin{corollary}\label{cor:singlebigd}
With the assumption of Theorem~\ref{thm:converge-to-zero}, there is no spurious locally optimal controller for large $\alpha$. That is, $K^\dagger(\alpha) = K^*(\alpha)$ for all large $\alpha$. 
\end{corollary}

\begin{proof}
  For any given $r>0$, all controllers in the ball $\mathcal{B}=\{K: \|K\| \leq r\}$ are stabilizing when $\alpha$ is large. As a result, stability constraints can be relaxed over $\mathcal{B}$. Furthermore, from Theorem~\ref{thm:converge-to-zero}, when $\alpha$ is large, all locally optimal controllers will be inside $\mathcal{B}$. From Theorem~\ref{thm:convex}, the objective function become convex over $\mathcal{B}$ for large enough $\alpha$. The observations imply local and global solutions coincide. 
\end{proof}

The theorems above rely on the ``damping property'' in Lemma~\ref{def:damping}. It is worth commenting that damping the system with $-I$ is almost the only continuation method for general system matrices $A$ that achieves the monotonic increasing of stable sets. Formally, 
\begin{theorem}
    When $n\geq 3$, for any $n$-by-$n$ real matrix $H$ that is not a multiple of $-I$, there exists a stable matrix $A$ for which $A+H$ is unstable. 
\end{theorem} 
The proof is given in the supplement. This theorem justifies the use of $-\alpha I$ as the continuation parameter. However, in a given system with structure, matrices other than $-I$ may be appropriate. 

\section{Numerical Experiments}\label{sec:numerical} 

In this section, we document various homotopy behaviors as the damping parameter $\alpha$ varies. The focus is on the evolution of locally optimal trajectories, which can be tracked by any local search methods. 
The experiments are performed on small-sized systems so the random initialization can find a reasonable number of distinct locally optimal solutions. Despite the small system dimension, the existence of many locally optimal solutions and their convoluted trajectories demonstrates what is possible in a theory of homotopy. 

The local search methods we used is the simplest projected gradient descent. At a controller $K^i$, we perform line search along the direction $\tilde{K}^i = - \nabla J(K) \circ I_S$. The step size is determined with backtracking and Armijo rule, that is, we select $s^i$ as the largest number in $\{\bar{s}, \bar{s}\beta, \bar{s}\beta^2, ...\}$ such that $K^i+s^i\tilde{K}^i$ is stabilizing while 
\[J(K^i+s^i\tilde{K}^i)<J(K^i)+\alpha s^i \langle \nabla J(K^i), \tilde{K}^i\rangle. \]
Our choice of parameters are $\alpha=0.001$, $\beta=0.5$, and $\bar{s}=1$. We terminate the iteration when the norm of the gradient is less than $10^{-3}$. 

\subsection{Systems with a large number of local minima}
We first consider the examples from \cite{fengExponentialNumberConnected}, where the feasible set is reasonably disconnected and admits many local minima. The system matrices are given by 
\begin{align}\label{eq:aeg}
A = \begin{bmatrix}
    -1 & 2 & 0 & 0\\
    -2 & 0 & 1 & 0 \\
     0 & -1 & 0 & 2\\ 
     0 & 0  & -2 & 0 \\
\end{bmatrix},
B &= \left[\begin{array} {cccc}
     0 & 1 & 0 & 0 \\
     -1 & 0 & 1 & 0 \\ 
     0 & -1 & 0 & 1 \\
     0 & 0 & -1 & 0\\ 
\end{array}\right],  D_0= I, \ I_\mathcal{S}=I. 
\end{align}
When the dimension $n$ is $4$, it is known that the set of stabilizing decentralized controllers has at least $5$ connected components. We sample the initial controllers from $N(0,1)$ and, after 1000 samples, obtain $5$ initial optimal solutions. We gradually increase the damping parameter from $0$ to $0.6$ with $0.002$ increment, and track the trajectories of locally optimal solutions by solving the newly damped system with the previous local optimal solution as the initialization. The evolution of the optimal cost and the distance from the best known optimal controller is plotted Figure~\ref{fig:expeg}. Notice that all sub-optimal local trajectories terminate after a modest damping $\alpha\approx0.2$. After that, the minimization algorithm always tracks a single trajectory. This illustrates the prediction of Corollary~\ref{cor:singlebigd}. Especially, if we start tracking a sub-optimal controller trajectory from $\alpha=0$, we will be on the better trajectory when $\alpha\approx 0.2$. At that time, if we gradually decrease $\alpha$ to zero, we obtain a stabilizing controller with a lower cost. 

\subsection{Experiments on Random Systems}
With the same initialization and optimization procedure, we perform the experiments with $3$-by-$3$ system matrices $A$ and $B$ randomly generated from the distribution $N(0,1)$. For 92 out of 100 samples we are not able to find more than one locally optimal trajectory. Examples with more than one local trajectories are listed below. All figures to the left plot the cost of locally optimal controllers. All figures to the right plot the distance of the locally optimal controllers to the controller with the lowest cost.  Note that the order of the cost of the trajectories may be preserved during the damping (Figure~\ref{fig:random1}) and may also be disrupted (Figure~\ref{fig:random2}). More than one trajectory may have the lowest cost during the damping (Figure~\ref{fig:random3}). 

\begin{figure}[htbp]
  \centering
  \includegraphics[width=0.4\textwidth]{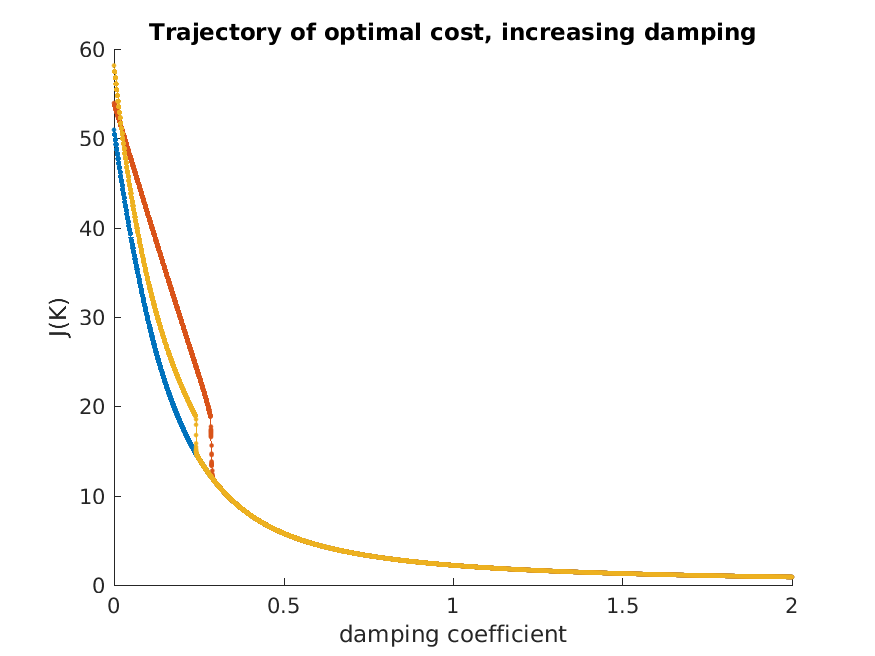}
  \includegraphics[width=0.4\textwidth]{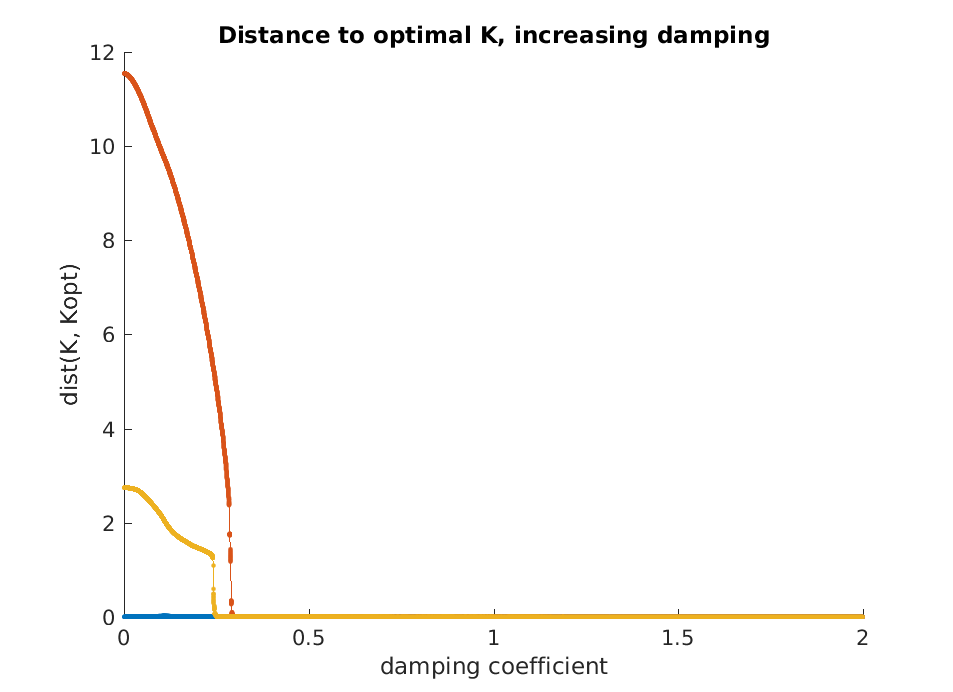}
  \caption{\label{fig:random1}Trajectory of a randomly generated system where the order of locally optimal controller is preserved.}
\end{figure}  
\begin{figure}[htbp]
  \centering
  \includegraphics[width=0.4\textwidth]{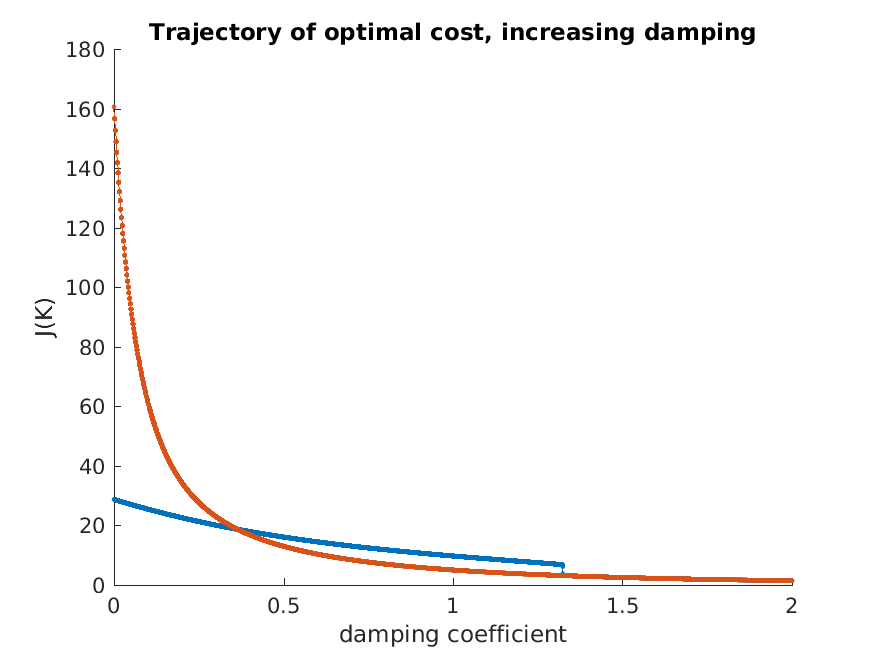}
  \includegraphics[width=0.4\textwidth]{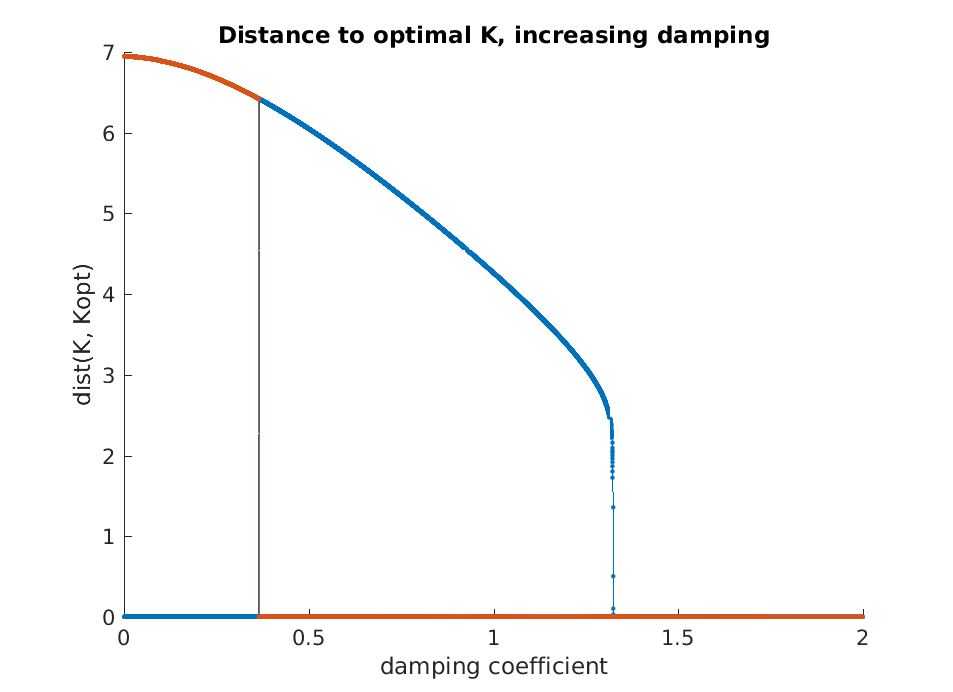}
  \caption{\label{fig:random2}Trajectories of a randomly generated system where the order of locally optimal controller is disrupted.}
\end{figure} 
\begin{figure}[htbp]
  \centering
  \includegraphics[width=0.4\textwidth]{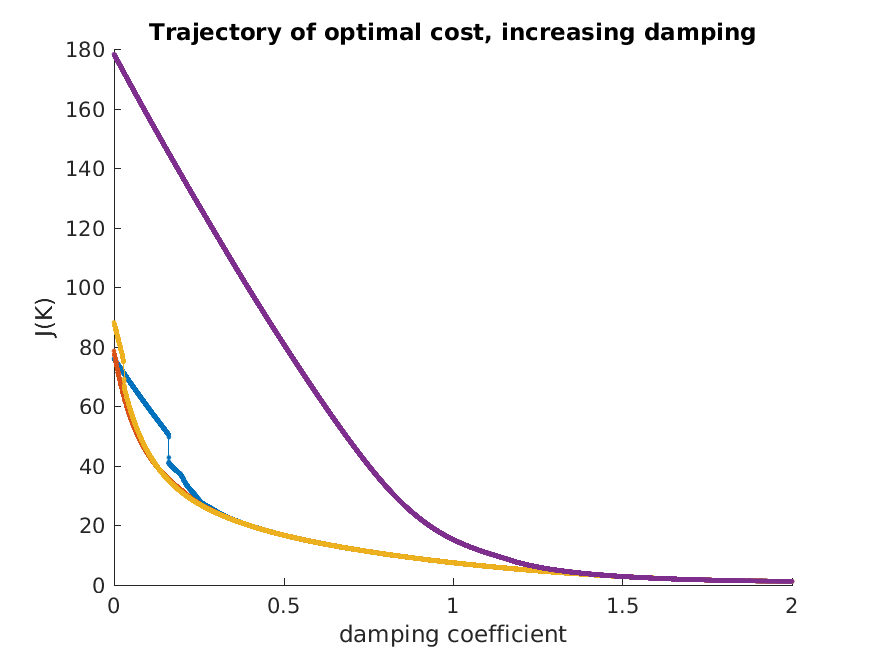}
  \includegraphics[width=0.4\textwidth]{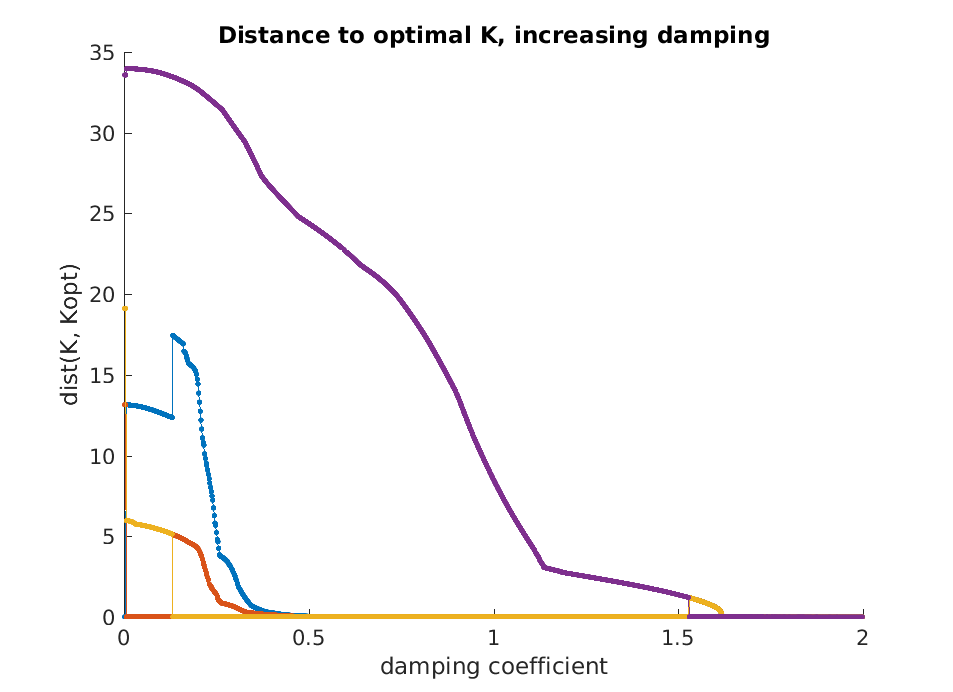}
  \caption{\label{fig:random3}Trajectory of a randomly generated system with a complicated behavior.}
\end{figure}   

Figure~\ref{fig:hysteresis} shows a hysteresis-like loop as the damping coefficient is first decreased and then increased. The trajectory of the controller first leads up to large cost and, the local search method escapes this local minimum to another one with a smaller cost. As the damping decreases, 
it returns where it starts along a different route. 
\begin{figure}[htbp]
    \centering
    \includegraphics[width=0.5\linewidth]{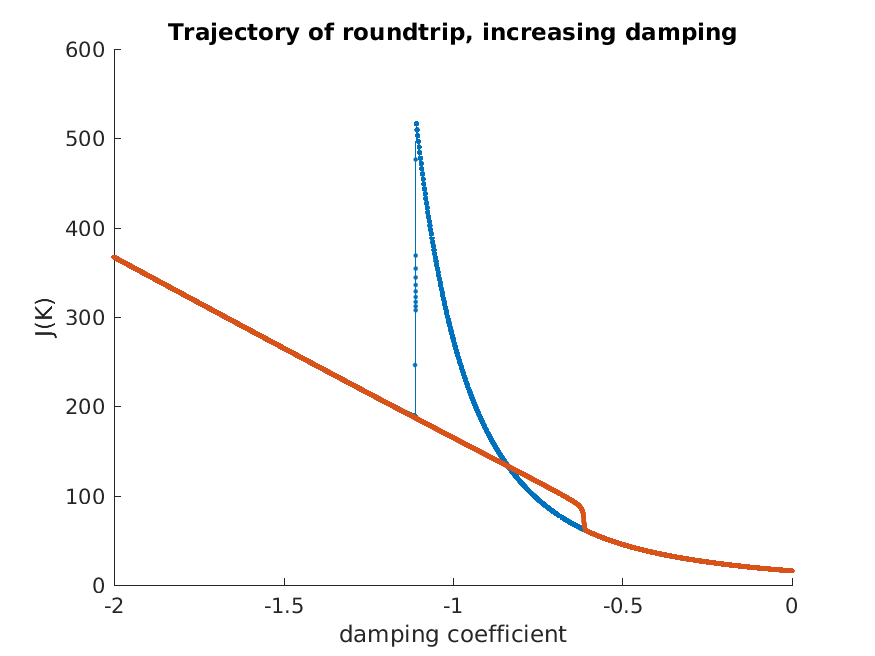}
    \caption{\label{fig:hysteresis}First decrease damping and then increase damping.}
\end{figure}

\section{Conclusion}\label{sec:conclusion}

This paper studied the trajectory of locally and globally optimal solution to the optimal decentralized control problem  as the damping of the decentralized control system varies. Asymptotic and continuity properties of trajectories are proved. The complicated phenomenon of continuation is illustrated with numerical examples. The fact that damping merges all locally optimal solutions is strong evidence that the idea of homotopy can be fruitfully used to improve locally optimal solutions. 

\subsubsection*{Acknowledgments}

The authors are grateful to Salar Fattahi and C\' edric Josz for their constructive comments and feedback. The author thanks Yuhao Ding for sharing the implementation of local search algorithms. 


\medskip

\small

\bibliographystyle{plain}
\bibliography{ZoteroLibrary} 

\appendix 
\section{Notions of continuity}
We recount the notion of upper and lower hemi-continuity and prove the continuity properties of the lower level-set map. The reader is referred to \cite{okRealAnalysisEconomic2007} for an accessible treatment. 

\begin{definition}
  The set valued map $\Gamma: A\to B$ is said to be upper hemi-continuous (uhc) at a point $a$ if for any open neighborhood $V$ of $\Gamma(a)$ there exists a neighborhood $U$ of $a$ such that $\Gamma(U)\subseteq V$. 
\end{definition}
If $B$ is compact, uhc is equivalent to the graph of $\Gamma$ being closed, that is, if $a_n\to a^*$ and $b_n \in \Gamma(a_n) \to b^*$, then $b^*\in \Gamma(a^*)$. 

\begin{definition}
  The set valued map $\Gamma: A\to B$ is said to be lower hemi-continuous (lhs) at a point $a$ if for any open neighborhood $V$ intersecting $\Gamma(a)$ there exists a neighborhood $U$ of $a$ such that $\Gamma(x)$ intersects $V$ for all $x\in U$. 
\end{definition}
Equivalently, for all $a_m\to a\in A$ and $b\in \Gamma(a)$, there exists $a_{m_k}$ subsequence of $a_m$ and a corresponding $b_k \in \Gamma(a_{m_k})$, such that $b_k \to b$. 

We prove the upper hemi-continuity of the lower level set map in Lemma~\ref{lem:lowerlevel} below. 
\begin{lemma}\label{lem:lowerlevel}
Given matrices $A, B$ and the objective cost $J(K, \alpha)$ that satisfies the damping property. Define 
  \begin{align*}
  \Gamma_M(\alpha) = \{K\in S: A-\alpha I + BK \text{ is stable and } J(K, \alpha) \leq M\}. 
  \end{align*}
Assume that $\Gamma_M(\alpha)$ is not empty for all $\alpha\geq0$ and a given $M>0$, then $\Gamma_M(\alpha)$ is an upper hemi-continuous set-valued map. 
\end{lemma}
\begin{proof}
  From~\cite{toivonenGloballyConvergentAlgorithm1985}, $\Gamma_M(\alpha)$ is compact for all $\alpha$. From the damping property, for any $\alpha < \beta$, we have $\Gamma_M(\alpha)\subseteq \Gamma_M(\beta)$.  Therefore, to characterize the continuity of $\Gamma$ at a $\alpha^*\geq0$, it suffices to consider the restricted map $\Gamma_M: [\alpha^*-\epsilon, \alpha^*+ \epsilon] \to \Gamma_M(\alpha^* + \epsilon)$ for some $\epsilon>0$, that is, to consider the range of $\Gamma_M$ to be compact. Therefore, the sequence characterization of uhc applies. Suppose $\alpha_i \to \alpha^*$, pick a sequence of $K_i \in \Gamma_M(\alpha_i)$ that converges to $K^*$. The continuity of $J(K, \alpha)$ implies  $J(K^*, \alpha^*) \leq M$. The fact that the cost is bounded implies $A-\alpha^* I + BK$ is stable. Since subspaces of matrices are closed, $K^*\in \mathcal{S}$. We have verified all conditions for $K^* \in \Gamma(\alpha^*)$, so $\Gamma_M$ is upper hemi-continuous. 
\end{proof}
The lower hemi-continuity of $\Gamma_M$ is more subtle. 
\begin{lemma}\label{lem:lowerlevellhc}
  At any given $\alpha^* \geq0$, $\Gamma_M(\alpha)$ is lower hemi-continuous at $\alpha^*$ except when $M\in \{J(K, \alpha^*): K\in K^\dagger(\alpha^*)\}$, which is a finite set of locally optimal costs. 
\end{lemma}
\begin{proof}
  Prove by contradiction, consider a sequence $\alpha_i\to \alpha^*$ and a $K^*\in \Gamma(\alpha^*)$, but there exists no subsequence of $\alpha_i$ and $K_i\in \Gamma(\alpha_i)$ such that $K_i\to K^*$. 
  We must have $J(K^*, \alpha^*)=M$ --- otherwise $J(K^*, \alpha_i)<M$ for large $i$ and, since the set of stabilizing controllers is open, $K^*\in \Gamma_M(\alpha_i)$ for large $i$. Furthermore, $K^*$ must be a local minimum of $J(K, \alpha^*)$ --- otherwise there exists a sequence $K_j \to K^*$ with $J(K_j, \alpha^*)<M$ and, by the continuity of $J$, there exists as sequence of large enough indices $n_j$ such that $J(K_j, \alpha_{n_j})<M$; the sequence $K_j \in \Gamma_M(\alpha_{n_j})$ converges to $K^*$. The argument above suggests that  $M$ belongs to the cost locally optimal controllers at $\alpha^*$. Because $J(K, \alpha^*)$ as a function over $K$ can be described as a linear function over an algebraic set, the value of local minimum is finite. 
\end{proof}


\section{Convergence of locally optimal controllers}
We prove the asymptotic properties of the locally optimal controllers in Section 4 of the main paper. 
\begin{theorem*}
  Suppose the sparsity pattern $I_S$ is block-diagonal with square blocks, and $R$ has the same sparsity pattern as $I_S$. Then all points in $K^\dagger(\alpha)$ converges to the zero matrix as $\alpha \to \infty$. Furthermore, $J(K, \alpha)\to 0$ as $\alpha\to\infty$ for all $K\in K^\dagger(\alpha)$. 
\end{theorem*}

\begin{proof}
Recall the expression of the objective function
\begin{equation}
\begin{aligned}
    J(K, \alpha) = & \bE \int_0^\infty \left[e^{-2\alpha t}\left(  \hat x^\top (t) Q \hat x(t) + \hat u^\top  (t) R \hat u(t) \right)\right] dt \\ 
  s.t.  \quad   & \hat{\dot {x}}(t) =  A \hat x(t) + B \hat u(t) \\ 
  & \hat u(t) = K \hat x(t), 
\end{aligned}\label{eq:appendix-damped-rep}
\end{equation}
and the first order necessary conditions
\begin{align}
& (A - \alpha I + BK)^\top  P_\alpha(K) +P_\alpha(K) (A - \alpha I + BK) + K^\top  RK + Q = 0 \label{eq:appendix-fonp}\\
& L_\alpha(K)(A-\alpha I  + BK)^\top  + (A-\alpha I +BK)L_\alpha(K) + D_0 = 0 \label{eq:appendix-fonl} \\ 
&  ((B^\top P_{\alpha}(K) + RK) L_\alpha(K))\circ I_S= 0\label{eq:appendix-stationary}\\
& K \circ I_\mathcal{S}=K \label{eq:appendix-sparsity}. 
\end{align}
Those first order conditions can be used to characterize the objective function 
\begin{equation}\label{eq:appendix-jpk}
  J(K, \alpha) = \tr(D_0 P_\alpha(K)). 
\end{equation}
  As $\alpha$ increases, some local solution may disappear, some new local solution may appear. The appearance cannot happen infinitely often because the equations \eqref{eq:appendix-fonp}-\eqref{eq:appendix-sparsity} are algebraic. Suppose when $\alpha\geq \alpha_0$, the number of local solutions does not change. The damping property ensures for $\beta > \alpha > \alpha_0$, 
  \begin{align*}
  \max_{K\in K^\dagger(\beta)} J(K, \beta) \leq \max_{K\in K^\dagger(\alpha)}J(K, \beta)
  \end{align*}
  The right hand side optimizes over a fixed, finite set of controllers and goes to zero as $\beta\to\infty$ from the formulation \eqref{eq:appendix-damped-rep} and the dominated convergence theorem. The left hand side, therefore, also converges to zero as $\beta\to\infty$. From \eqref{eq:appendix-jpk} and the assumption that $D_0$ is positive definite, $\|P_{\beta}(K)\|\to 0$ for all $K\in K^\dagger(\beta)$ as $\beta\to\infty$. 

  The assumption on sparsity allows the expression of the locally optimal controllers in \eqref{eq:appendix-stationary} as 
   \[K = -R^{-1}((B^\top  P_\alpha(K) L_\alpha(K)) \circ I_S) (L_\alpha(K)\circ I_S)^{-1}.\] Especially we can bound 
  \begin{align*}
   \|BK\| \leq \|BR^{-1}B^\top  P_\alpha(K) L_\alpha(K) \| \lambda_{\min} (L_\alpha(K))^{-1}. 
  \end{align*}
  Pre- and post- multiply \eqref{eq:appendix-fonl} by $L_\alpha(K)$'s unit minimum eigenvector $v$,  
  \begin{align}
  \lambda_{\min}(L_\alpha(K))(2a - 2 v^\top (A+BK)v) = v^\top D_0 v. 
  \end{align} 
  Therefore
  \begin{align}
  \lambda_{\min}(L_\alpha(K)) & \geq \frac{\lambda_{\min}(D_0)}{2 \alpha + 2 \|A + BK\|} \\ 
   & \geq \frac{\lambda_{\min}(D_0)}{2 \alpha + 2 \|A\| + 2\|BK\|}  \\ 
   & \geq \frac{\lambda_{\min}(D_0)}{2 \alpha + 2 \|A\| + 2 \|BR^{-1}B^\top  P_\alpha(K) L_\alpha(K) \| \lambda_{\min} (L_\alpha(K))^{-1}. }\label{eq:labelow}
  \end{align}
  This simplifies to 
  \begin{align}\label{eq:lmin}
  \lambda_{\min}(L_\alpha(K)) \geq \frac{\lambda_{\min}(D_0) - 2 \|BR^{-1}B^\top  P_\alpha(K) L_\alpha(K) \|}{(2 \alpha + 2 \|A\|)}
  \end{align}
  Take the trace of \eqref{eq:appendix-fonl} and consider the estimate
  \begin{align*} 
  2n \|A\|\|L_\alpha\| + \tr(D_0) & \geq 2 \|A\| \tr(L_\alpha) + \tr (D_0) \\ 
   & \geq 2 \alpha \tr(L_\alpha) + 2 \tr(BR^{-1}((B^\top  P_\alpha L_\alpha) \circ I_S) (L_\alpha\circ I_S)^{-1}L_\alpha) \\ 
  & \geq 2\alpha \tr(L_\alpha) - 2 \|BR^{-1}((B^\top  P_\alpha L_\alpha) \circ I_S)\| \tr((L_\alpha\circ I_S)^{-1}L_\alpha) \\ 
  & = 2\alpha \tr(L_\alpha) - 2 \|BR^{-1}((B^\top  P_\alpha L_\alpha) \circ I_S)\| n \\ 
  & \geq 2\alpha \|L\| - 2n \|BR^{-1}\|\|B^\top\|\|  \|P_\alpha\|\| L_\alpha \|, 
  \end{align*}
  where for clarity $L_\alpha$ denotes $L_\alpha(K)$ and $P_\alpha$ denotes $P_\alpha(K)$. The second and the third inequalities use the fact that $|\tr(AL)|\leq \|A\| \tr(L)$ for a positive definite matrix $L$ and any matrix $A$. This estimate, combined with previous argument that $\|P_\alpha\|\to 0$, concludes $\|L_\alpha\|\to 0$. We also obtain from the inequality that 
  \begin{align}\label{eq:lmax}
  \|L_\alpha\| \leq \frac{\tr(D_0)}{2a - 2n \|A\| - 2n \|BR^{-1}\|\|B^\top\|\|  \|P_\alpha\|}, 
  \end{align}
  for small enough $P_\alpha$. 
  Combining \eqref{eq:lmin} and \eqref{eq:lmax}
  \begin{align*}
  \|K\| 
  & \leq  \|R^{-1}\|\cdot \|(B^\top  P_\alpha L_\alpha) \circ I_S \|\cdot\| (L_\alpha\circ I_S)^{-1}\| \\ 
  & \leq \|R^{-1}\|\cdot \|B^\top  \| \cdot \| P_\alpha\| \cdot \| L_\alpha\|\cdot\|\lambda_{\min}(P_\alpha)^{-1} \\ 
  & \leq  \|R^{-1}\|\cdot \|B^\top  \| \cdot \| P_\alpha\| \frac{\tr(D_0)}{2\alpha - 2n \|A\| - 2n \|BR^{-1}\|\|B^\top\|\|  \|P_\alpha\|} \frac{(2 \alpha + 2 \|A\|)}{\lambda_{\min}(D_0) - 2 \|BR^{-1}B^\top  P_\alpha L_\alpha\|}, 
  \end{align*}
  which converges to $0$ as $\alpha\to \infty$. 
\end{proof}

\section{The Positive Definiteness of Hessian}
\begin{theorem*}
For any given $r>0$, the Hessian matrix $\nabla^2 J(K, \alpha)$ is positive definite over $\|K\| \leq r$ for all large $\alpha$. 
\end{theorem*} 

\begin{proof}
The proof requires the vectorized Hessian formula given in Lemma~3.7 of \cite{rautertComputationalDesignOptimal1997}, restated below. 
\begin{lemma}[\cite{rautertComputationalDesignOptimal1997}]
  Define $j_\alpha: \bR^{m\cdot n} \to \bR$ by $j_\alpha(vec(K)) = J(K, \alpha)$. The Hessian of $j_\alpha$ is given by the formula 
  \begin{align*}
  H_\alpha(K) = 2 \left\{ \left(L_\alpha(K)\otimes R\right) + G_\alpha(K)^\top + G_\alpha(K)\right\}, 
  \end{align*}
where
  \begin{align*}
  G_\alpha(K) = [I\otimes (B^TP_\alpha(K) + RK)]\left[I\otimes (A-\alpha I +BK) + (A-\alpha I +BK)\otimes I\right]^{-1} (I_{n,n} + P(n,n))[L_\alpha(K)\otimes B]
  \end{align*}
and $P(n,n)$ is an $n^2\times n^2$ permutation matrix. 
\end{lemma}

We show that $H_\alpha(K)$ in the lemma is positive definite for any fixed $K$ when $\alpha$ is large. Recall the definition of $L_\alpha$ and $K_\alpha$. 
\begin{align}
  L_\alpha(K)(A-\alpha I  + BK)^T + (A-\alpha I +BK)L_\alpha(K) + D_0 = 0, \label{eq:lalpha}\\ 
  P_\alpha(K)(A-\alpha I +BK) + (A-\alpha I +BK)^TP_\alpha(K) +K^TRK + Q = 0. 
\end{align}
With triangle inequality
\begin{align*}
2 \alpha \|L_\alpha(K)\| & \leq \|D_0\| + 2 \|A + BK \| \|L_\alpha(K)\| \\ 
2 \alpha \|P_\alpha(K)\| & \leq \|Q\| + 2 \|A + BK \| \|P_\alpha(K)\| + \|R\|\|K\|^2\\ 
\end{align*}
which means $\|P_\alpha(K)\|\to 0$ and $\|L_\alpha(K)\| \to 0$ as $\alpha\to \infty$. The minimum eigenvalue of $L_\alpha(K)$ can be bounded similarly: let $v$ be the unit eigenvector of $L_\alpha(K)$ corresponding to $\lambda_{\min} (L_{\alpha}(K))$, pre- and post- multiply \eqref{eq:lalpha} by $v$, we obtain 
\begin{align}
\lambda_{\min}(L_{\alpha}(K)) \geq \frac{v^T D_0 v}{2 \alpha - 2 v^T (A+BK) v} \geq \frac{\lambda_{\min}(D_0)}{2 \alpha + 2 \|A + BK\|}. \label{eq:l0below}
\end{align}
The first Hessian term $L_\alpha(K)\otimes R$ can bounded from below with \eqref{eq:l0below}
\begin{align*}
\lambda_{\min} \left(L_\alpha(K) \otimes R \right) = \lambda_{\min} (L_\alpha(K)) \lambda_{\min}(R) \geq \frac{\lambda_{\min}(D_0) \lambda_{\min}(R)}{2 \alpha + 2 \|A + BK\|}. 
\end{align*}
We bound the norm of the second and the third Hessian term $\|G_\alpha(K)\|$ as follows, where $\lesssim$ hides constants that do not depend on $\alpha$. 
\begin{align*}
\|G_\alpha(K)\| 
& \leq \|I\otimes (B^TP_\alpha(K) + RK)\| \\ 
& \cdot \|\left[I\otimes (A-\alpha I + BK) + (A-\alpha I + BK)\otimes I\right]^{-1}\|\cdot \| (I_{n,n} + P(n,n))[L_\alpha(K)\otimes B] \| \\ 
& \lesssim (-\lambda_{\max}\left(I\otimes (A-\alpha I + BK) + (A-\alpha I + BK)\otimes I\right))^{-1} \|L_\alpha(K)\| \\ 
& \lesssim (2\alpha)^{-1}\|L_\alpha(K)\|. 
\end{align*}
Comparing the two estimates above, we find the first term dominates the two following terms with large $\alpha$, uniformly over bounded $K$. The Hessian $H_\alpha(K)$ is therefore positive definite over bounded $K$ when $\alpha$ is large. The conclusion carries over to the Hessian of the decentralized controller, which is a principal sub-matrix of the Hessian of the centralized controller. 
\end{proof}

\section{The uniqueness of the continuation direction} 
This section aims to prove the following result
\begin{theorem*}
    When $n\geq 3$, for any $n$-by-$n$ real matrix $H$ that is not a multiple of $-I$, there exists a stable matrix $A$ for which  $A+H$ is unstable. 
\end{theorem*} 

Define the set of stable directions
  \begin{align}
  \mathcal{H} = \{H: A+tH \text{ is stable whenever } A \text{ is stable and } t\geq 0\}, \label{eq:define-H}
  \end{align}
where $A$ and $H$ are $n$-by-$n$ real matrices. 

\begin{lemma} \label{lem:Hdiagonal}
  All matrices in $\mathcal{H}$ is similar to a diagonal matrix with non-positive diagonal entries. Especially, they cannot have complex eigenvalues. 
\end{lemma}

\begin{proof}
    When $t$ is large, $A+tH$ is a small perturbation of $tH$, hence the eigenvalues of $H$ has to be in the closed left half plane. 
  With a suitable similar transform assume $H$ is in real Jordan form. First consider the case of two by two matrices, and we denote the matrices by $H_2$ and $A_2$. Assume for contradiction that $H_2$ is not diagonalizable. The non-diagonal real Jordan form of $H_2$ has the following possibilities:
  \begin{itemize}
    \item 
    $H_2 = \begin{bmatrix} h & 1 \\ 0 & h \end{bmatrix}$, where $H_2$ has real eigenvalues $h<0$. 
    Pick $A_2 = \begin{bmatrix} 4h & -2 \\ 10h^2 & -3h \end{bmatrix}$, which is stable because $tr(A_2)=h<0$ and $\det(A_2) = 8 h^2 > 0$. We have $A_2+tH_2 = \begin{bmatrix}h t + 4h by& t - 2 \\ 10h^2 &  ht- 3h\end{bmatrix}$, whose stability criterion $
      tr (A_2+tH_2) < 0$ and $
      \det (A_2+tH_2) > 0
    $
  amounts to 
  \begin{align*}
      2h t + h < 0  \\ 
      h^2 (t^2 -9 t +8) > 0, 
  \end{align*}
  or equivalently $t\in (-1/2, 1) \cup (8, +\infty)$. Especially when $t=2$, $A_2 + tH_2$ is not stable.
    \item  
    $H_2 = \begin{bmatrix} 0 & 1 \\ 0 & 0 \end{bmatrix}$. 
    Pick a stable matrix $A = \begin{bmatrix} -1 & 0 \\ 1 & -1  \end{bmatrix}$. $A+tH$ is not stable when $t=2$. 
    \item $H_2 = \begin{bmatrix} 0 & f \\ -f & 0 \end{bmatrix}$, where $f>0$, 
    Pick $A = \begin{bmatrix} -1 & -4\\ 1 & -1 \end{bmatrix} $, $A+\frac{2}{f} H_2 = \begin{bmatrix} -1 & -2\\ -1 & -1 \end{bmatrix} $ is not stable. 
    \item $H_2 = \begin{bmatrix} h & f \\ -f & h \end{bmatrix}$, where $h<0$ and $f>0$. By rescaling assume $f=1$. Consider the following matrix function 
      \begin{align}
        G(t) = \begin{bmatrix}
          0 & \frac12 + (u+w)h  \\ -\frac12 + (u-w)h & h 
        \end{bmatrix} + t \begin{bmatrix}
          h & 1 \\ -1 & h 
        \end{bmatrix}
      \end{align}
    We have 
    \begin{align*}
      tr(G(t)) & = h+2ht \\ 
      \det(G(t)) & = (1+h^2) t^2 + (1+h^2 + 2hw) t + h^2 (w^2 - u^2)  + hw + \frac14. 
    \end{align*}
    Espeically, 
    \begin{align*}
         tr(G(-\frac12)) & = 0 \\ 
         \frac{d}{dt} tr G(t) & = 2h \\ 
         \det(G(-\frac12)) &= h^2(-\frac14 - u^2 + w^2) \\ 
         \left. \frac{d}{dt} \det G(t) \right|_{t=-\frac12} & = 2 hw \\ 
    \end{align*}
    Hence as long as 
    \begin{align}
     w > 0  \text{ and } -\frac14 - u^2 + w^2 > 0 \label{eq:H2stab1}
    \end{align}
    for small enough $\epsilon>0$, $A_2 = G(-\frac12+\epsilon)$ is a stable matrix and there will be matrices $G(t)$ with $t>-\frac12$ whose trace is negative and whose determinant is smaller. Consider the minimal value the determinant can take
    \begin{align*}
    \det G\left(-\frac12 -\frac{hw}{1+h^2}\right) & = h^2\left(-\frac14 -u^2 + \frac{h^2}{1+h^2}w^2\right)
    \end{align*}
    which means when 
    \begin{align}
    -\frac14 -u^2 + \frac{h^2}{1+h^2}w^2 < 0 \label{eq:H2stab2}
    \end{align}
    The matrix $G(t)$ with $t=-\frac12 -\frac{hw}{1+h^2}$ is unstable. There certainly exist $u$ and $w$ that satisfies \eqref{eq:H2stab1} and \eqref{eq:H2stab2}. 
   \end{itemize} 
  For general $n$, $H$'s real Jordan form is an block upper-triangular matrix
  \begin{align*}
  H = \begin{bmatrix}
    H_2 & * \\ 0 & * 
  \end{bmatrix}
  \end{align*}
  where $H_2$ can take the four possibilities mentioned above. We take the corresponding stable $A_2$ constructed above, which has the property that $A_2 + t_0 H_2$ is not stable for some $t_0>0$. Form the block diagonal matrix 
  \begin{align*}
  A = \begin{bmatrix}
    A_2 & 0 \\ 0 & -I 
  \end{bmatrix}
  \end{align*}
  Then $A$ is stable, while $A+t_0H = \begin{bmatrix}
    A_2 + t_0H_2 & * \\ 0 & * 
  \end{bmatrix}$ is not stable. 
\end{proof}

We can strengthen the argument above and further characterize $\mathcal{H}$ in the case $n\geq 3$. 
\begin{lemma} \label{lem:norank1}
  When $n\geq 3$, the set of stable directions $\mathcal{H}$ does not contain any matrices of rank $1$, $2$, \ldots, $n-2$. 
\end{lemma}

\begin{proof} 
  From lemma \ref{lem:Hdiagonal}, we only need to consider the case where $H$ is diagonal with negative diagonal entries. Assume there is a rank one matrix $H \in \mathcal{H}$, write \begin{align*}
  H = \begin{bmatrix}
    H_3 & 0 \\ 0 & *
  \end{bmatrix}, 
  \end{align*}
  where $H_3 = \diag(-1, 0, 0)$. This is possible with the rank assumption. We will construct a stable $3$-by-$3$ matrix $A_3$, such that there is some $t_0>0$ that makes $A_3+t_0H_3$ unstable, and then carry the instability to $A + t_0 H$ with the extended matrix \[A=\begin{bmatrix}
    A_3 & 0 \\ 0 & -I
  \end{bmatrix}. \]
  From \cite{fengExponentialNumberConnected}, the set 
  \begin{align*}
     T = \left\{t: \begin{bmatrix}
      0 &1 &0 \\ 0 & 0 & 1 \\ 5 & 1 & -1 
    \end{bmatrix} + t 
    \begin{bmatrix}
  0 \\ 0 \\ -1
  \end{bmatrix}
  \begin{bmatrix}
      0.85 & 0.2 &0.2
  \end{bmatrix} \text{ is stable }\right\}
  \end{align*}
  has two disconnected components. Consider the Jordan decomposition of the matrix 
  \begin{align*}
   \begin{bmatrix}
  0 \\ 0 \\ -1
  \end{bmatrix}
  \begin{bmatrix}
      0.85 & 0.2 &0.2
  \end{bmatrix} = P \diag (-0.2, 0, 0) P^{-1}, 
   \end{align*}
   where $P$ is some invertible matrix. Write 
   \begin{align*}
    G(t) = 5 P^{-1} \begin{bmatrix}
      0 &1 &0 \\ 0 & 0 & 1 \\ 5 & 1 & -1 
    \end{bmatrix}P  + t \times \diag(-1, 0, 0). 
   \end{align*}
   After this similar transform, the set $T$ can be written with $G(t,0)$. 
   \begin{align*}
   T = \{t: G(t) \text{ is stable}\}
   \end{align*}
   Since $T$ is disconnected there exists some $t_1 < t_2$ such that $G(t_1)$ is stable, while $G(t_2)$ is unstable with some eigenvalue in the right half plane. Setting $A_3 = G(t_1)$ and $t_0 = t_2 - t_1$ completes the proof. 
\end{proof}

Since we can perturb the direction and make $H$ full-rank, the fact that $H$ has rank one is not the substantial property. This is indeed the case. 

\begin{lemma}
When $n\geq 3$, $\mathcal{H} = \{-\lambda I , \lambda \geq 0 \}$. 
\end{lemma}
  \begin{proof}
From lemma~\ref{lem:Hdiagonal}, we only need to consider the case where $H$ is diagonal with negative diagonal entries. Write \begin{align*}
  H = \begin{bmatrix}
    H_3 & 0 \\ 0 & *
  \end{bmatrix}, 
  \end{align*}
  where $H_3 = \diag(h_1, h_2, h_3)$. The diagonal entries $h_i, i=1,2,3$ are non-positive and not all equal. We will construct an $A_3$ and a corresponding $t_0$ such that $A_3$ is stable while $A_3+t_0H_3$ is not stable, and extend to the general $A$ as in Lemma~\ref{lem:norank1}. The case where $H_3$ has rank $1$ has been considered in Lemma~\ref{lem:norank1}. We show the remaining rank is impossible. Without loss of generality we rescale $H_3$ and assume $h_1=-1$. 
  \begin{itemize}
    \item $H_3 = \diag(-1, h_2, 0)$, where $h_2 <0$.  Consider the matrix function 
    \begin{align*}
     G(t)= \begin{bmatrix}
      0 & -1 & 0 \\
      0 & 0 & -h_2 \\ 
      2 & 1 & 0 
    \end{bmatrix} + t H_3 = \begin{bmatrix}
      -t & -1 & 0 \\
      0 & t h_2 & -h_2 \\ 
      2 & 1 & 0 
    \end{bmatrix}. 
    \end{align*}
    The characteristic polynomial of $G(t)$ is \[\phi_{G(t)}(x) = x^3 + (t-th_2) x^2 + (h_2 - t^2 h_2) x + (t-2)h_2. \] The Routh-Hurwitz Criterion insists 
    \begin{align*}
      t(1-h_2) & > 0 \\ 
      (t-2)h_2 & > 0 \\ 
      t(1-h_2)  h_2(1-t^2) & >(t-2)h_2. 
    \end{align*}
    which is simplified with $h_2 < 0$ to 
    \begin{align}
    & 0 < t < 2   \label{eq:H3rd-t-range}\\ 
    & (1-h_2) t^3 + th_2 -2 > 0 \label{eq:H3rd-t-cubic}. 
    \end{align}
    Especially, when $t=\frac32$, \eqref{eq:H3rd-t-cubic} simplifies to the obvious expression $\frac18(11-15h_2)>0$. when $t=3$, \eqref{eq:H3rd-t-range} implies $G(t)$ is not stable. Setting $A_3=G(\frac32)$ and $t_0 = \frac32$ concludes the proof. 

    \item $H_3 = \diag(-1, h_2, h_3)$, where without loss of generally we assume 
    \begin{equation}
      -1 \leq h_2, h_3<0, \text{ and one of them is not $-1$}.  \label{eq:h3rd-sign-condition}
    \end{equation}
    
    Consider the matrix  
    \begin{align*}
    G(t)= \begin{bmatrix}
      0 & -1 & 0 \\
      0 & 0 & h_2 \\ 
      ah_3 & h_3 & 0 
    \end{bmatrix} + t H_3 = \begin{bmatrix}
      -t & -1 & 0 \\
      0 & t h_2 & h_2 \\ 
      ah_3 & h_3 & th_3
    \end{bmatrix}
    \end{align*}
    Its Routh-Hurwitz Criterion insists 
    \begin{align}
    t  & > 0 \nonumber \\ 
    f_1(t) = a - t + t^3 &  > 0 \label{eq:cubic1}\\ 
    f_2(t) = -ah_2 h_3 + th_2 h_3(h_2 + h_3) + t^3 (1-h_2)(1-h_3)(-h_2-h_3)& > 0 \label{eq:cubic2}
    \end{align}
    We claim that when 
    \begin{align}
    \sqrt\frac{h_2 h_3(h_2 + h_3)^2}{(-h_2 - h_3 + h_2h_3)^3} < a < \sqrt\frac{4}{27} \label{eq:h3rd-condition}
    \end{align}
    the set of $t$ that satisfy Routh-Hurwitz Criterion is disconnected. To see this, write the positive local minimum of $f_1(t)$ in \eqref{eq:cubic1} as $t_1 = \sqrt{\frac13}$, and write the positive local minimum of $f_2(t)$ in \eqref{eq:cubic2} as $t_2 = \sqrt{\frac{h_2h_3}{3(1-h_1)(1-h_2)}}$. The condition \eqref{eq:h3rd-sign-condition} ensures that $t_1 < t_2$ and the condition \eqref{eq:h3rd-condition} ensures that $f_1(t_1)$ and $f_2(t_2)$ are negative. Furthermore, consider $t_0 = a \frac{h_2 + h_3 - h_2h_3}{h_2 + h_3}$, which is the root of $(1-h_2)(1-h_3)(-h_2-h_3) f_1(t)-f_2(t)$. It holds that $t_1 < t_0 < t_2$ and both $f_1(t_0)$ and $f_2(t_0)$ are positive, which implies that the positive intersection $f_1(t)$ and $f_2(t)$ are positive. 
    We conclude that when $t=t_0$, the matrix $G(t_0)$ is stable, and when $t$ is large, $G(t)$ is again stable. Yet when $t=t_2 \in (t_0, \infty)$, the matrix $G(t_2)$ is not stable. 
  \end{itemize}
\end{proof}

\end{document}